\documentclass{amsproc}
\usepackage{eurosym}
\usepackage{amssymb}
\usepackage{amsfonts}
\usepackage{cmtt}

\theoremstyle{plain}
\newtheorem{theorem}{Theorem}
\newtheorem{corollary}[theorem]{Corollary}
\newtheorem{proposition}[theorem]{Proposition}
\numberwithin{equation}{section}
\DeclareMathOperator{\ch}{ch}

\begin{document}

\title{List coloring of matroids and base exchange properties}

\author[Micha\l\ Laso\'{n}]{Micha\l\ Laso\'{n}}

\dedicatory{\upshape
Institute of Mathematics of the Polish Academy of Sciences,\\ ul.\'{S}niadeckich 8, 00-656 Warszawa, Poland\\ \textmtt{michalason@gmail.com}
\smallskip
\\ Theoretical Computer Science Department,\\ Faculty of Mathematics and
Computer Science, Jagiellonian University,\\ ul.{\L}ojasiewicza 6, 30-348 Krak\'{o}w, Poland
\\ \textmtt{mlason@tcs.uj.edu.pl}}

\thanks{Research partially supported by the Polish National Science Centre grant no. 2011/03/N/ST1/02918. The paper was completed during author's stay at Freie Universit\"at Berlin in the frame of Polish Ministry ``Mobilno\'s\'c Plus'' program.}

\begin{abstract}
A coloring of a matroid is an assignment of colors to the elements of its ground set. We restrict to proper colorings -- those for which elements of the same color form an independent set. Seymour proved that a $k$-colorable matroid is also colorable from any lists of size $k$.

We generalize this theorem to the case when lists have still fixed sizes, but not necessarily equal. For any fixed size of lists assignment $\ell$, we prove that, if a matroid is colorable from a particular lists of size $\ell$, then it is colorable from any lists of size $\ell$. This gives an explicit necessary and sufficient condition for a matroid to be list colorable from any lists of a fixed size.

As an application, we show how to use our condition to derive several base exchange properties.
\end{abstract}

\maketitle

\section{Introduction}

Let $M$ be a matroid on a ground set $E$ (we refer the reader to \cite{Ox92} for a background of matroid theory). A \emph{coloring} of $M$ is an assignment of colors to the elements of $E$. In analogy to graph theory we say that a coloring is \emph{proper} if elements of the same color form an independent set in the matroid. Via this correspondence one can define for matroids all chromatic parameters studied for graphs. 

The \emph{chromatic number} of a loopless matroid $M$, denoted by $\chi(M)$, is the minimum number of colors in a proper coloring of $M$. For instance, if $M$ is a graphic matroid obtained from a graph $G$, then $\chi(M)$ is the least number of colors needed to color edges of $G$ so that no cycle is monochromatic. This number is known as the \emph{arboricity} of the underlaying graph $G$. 

For matroids the chromatic number can be easily expressed in terms of rank function. Extending a theorem of Nash-Williams \cite{Na64} for graph arboricity, Edmonds \cite{Ed65} proved the following formula
\begin{equation*} 
\chi(M)=\max_{\emptyset\neq A\subset E}\left\lceil\frac{\left\vert A\right\vert}{r(A)}\right\rceil,
\end{equation*}
where $r$ is the rank function of a loopless matroid $M$ on a ground set $E$. Some game-theoretic versions of the chromatic number of a graph were already studied for matroids, see \cite{La12,La13}.

In this note we study list coloring of matroids. The concept of list coloring was initiated for graphs by Vizing \cite{Vi76}, and independently by Erd\H{o}s, Rubin and Taylor \cite{ErRuTa80}. Let us recall its definition in the matroid setting.

Suppose each element $e$ of the ground set $E$ of a matroid $M$ is assigned with a set (a \emph{list}) of colors $L(e)$. By the \emph{size} of \emph{list assignment} (or simply \emph{lists}) $L$ we mean the function $\ell$ satisfying $\ell(e)=\left\vert L(e)\right\vert$ for each $e\in E$. We say that matroid $M$ is \emph{colorable from lists} $L$ if there exists a proper coloring of $M$, such that each element receives a color from its list. The \emph{list chromatic number} (sometimes called \emph{choice number}) of a loopless matroid $M$, denoted by $\ch(M)$, is the minimum number $k$, such that $M$ is colorable from any lists of size at least $k$. 

Clearly, $\ch(M)\geq\chi(M)$. For graphs in general, the inequality between corresponding parameters is strict ($\ch(G)$ is not even bounded by a function of $\chi(G)$). A celebrated result of Galvin \cite{Ga95} asserts that for the line graph of a bipartite graph there is an equality. Surprisingly, Seymour \cite{Se98} proved that actually equality holds for all matroids. 

\begin{theorem}\label{TheoremSeymour}
For every loopless matroid $M$ there is an equality $\ch(M)=\chi(M)$.
\end{theorem}

Seymour's theorem can be rephrased by saying that the following conditions are equivalent:
\begin{enumerate}
\item matroid $M$ is colorable from the lists $L(e)=\{1,\dots,k\}$,
\item matroid $M$ is colorable from any lists of size $k$.
\end{enumerate}

Our main result is a generalization of Seymour's Theorem \ref{TheoremSeymour} to the setting, where sizes of lists are still fixed, but not necessarily equal. 

\begin{theorem}\label{main} Let $M$ be a matroid and let $\ell$ be a lists size function. Then the following conditions are equivalent:
\begin{enumerate}
\item matroid $M$ is colorable from the lists $L_{\ell}(e)=\{1,\dots,\ell(e)\}$,
\item matroid $M$ is colorable from any lists of size $\ell$.
\end{enumerate}
\end{theorem}

As a corollary we get a strengthening of Seymour's Theorem \ref{TheoremSeymour}. Namely, a $k$-colorable matroid is also colorable from any lists of fixed size varying between $1$ and $k$, and average size at most $\frac{k+1}{2}$.

\begin{corollary}
Let $M$ be a $k$-colorable matroid, and let $I_{1},\dots,I_{k}$ be a partition of its ground set into independent sets (color classes). Define lists size function by $\ell(e)=i$ for elements $e\in I_{i}$. Then matroid $M$ is colorable from any lists of size $\ell$.
\end{corollary}

In the last section we make a link between list coloring and base exchange properties. We use our theorem as a tool to obtain easily several such properties. The idea is to choose suitable lists, such that existence of a proper coloring guarantees a particular exchange property. The crucial point is that lists may have different sizes (ex. if size of a list is $1$, then a color is already determined).

\section{Proof of Theorem \ref{main}}\label{Section2}

\begin{proof}
Clearly, condition $(1)$ follows from $(2)$. We argue the opposite implication. Let $L$ be a fixed list assignment of size $\ell$. Without loss of generality we can assume that all lists $L(e)$ are subsets of a finite set of integers $\{1,\dots,d\}$. Let us denote $Q_i=\{e\in E:i\in L(e)\}$, and respectively $Q^{\ell}_{i}$ for lists $L_{\ell}$. 

Consider matroids $M_1,\dots,M_d$, with $M_i$ equal to the restriction $M\vert_{Q_i}$ of $M$ to the set $Q_i$ (with the ground set trivially extended to $E$). It is straightforward that a proper coloring from lists $L$ exists if and only if it is possible to partition the ground set $E$ into subsets $I_1,\dots,I_d$ with $I_i$ independent in the matroid $M_i$ ($I_i$ is the color class of $i$). By the Matroid Union Theorem (see \cite{Ox92}) such a partition exists if and only if for every subset $A\subset E$ there is an inequality
\begin{equation*}
r(A\cap Q_1)+\cdots+r(A\cap Q_d)\geq\left\vert A\right\vert.
\end{equation*}

Analogously, a proper coloring from lists $L_{\ell}$ exists if and only if for every subset $A\subset E$ there is an inequality
\begin{equation*}
r(A\cap Q^{\ell}_1)+\cdots+r(A\cap Q^{\ell}_d)\geq\left\vert A\right\vert.
\end{equation*}

Thus, to prove that condition $(1)$ implies $(2)$ it is enough to show that for every subset $A\subset E$ there is an inequality
\begin{equation}\label{eq0}
r(A\cap Q_1)+\cdots+r(A\cap Q_d)\geq r(A\cap Q^{\ell}_1)+\cdots+r(A\cap Q^{\ell}_d).
\end{equation}

Notice that $\bigcup_iQ_i$ and $\bigcup_iQ^{\ell}_i$ are equal as multisets, since both $L$ and $L_\ell$ are list assignments of size $\ell$ (each $e\in E$ belongs to each of the unions exactly $\ell(e)$ times). We will show that the inequality $(\ref{eq0})$ is satisfied for any sets $Q_i$ satisfying $\bigcup_iQ_i=\bigcup_iQ^{\ell}_i$ as multisets. The proof is by induction on the number of pairs of sets $Q_k,Q_l$ such that $Q_k$ and $Q_l$ are incomparable in the inclusion order (none is contained in the other).

If the number of such pairs is zero, then $Q_i$ are linearly ordered by inclusion. Let us reorder them in such a way that $Q_1\supset Q_2\supset\dots\supset Q_d$. Then the equality $\bigcup_iQ_i=\bigcup_iQ^{\ell}_i$ implies that $Q_1=Q^{\ell}_1$, $Q_2=Q^{\ell}_2$, $\dots$, $Q_d=Q^{\ell}_d$, so the inequality $(\ref{eq0})$ is in fact an equality.

Suppose now that there exists a pair of sets $Q_k,Q_l$ incomparable in the inclusion order. Replace them in the family $\{Q_i\}_{i=1,\dots,d}$ by sets $Q_k\cup Q_l$ and $Q_k\cap Q_l$ to obtain a family $\{Q'_i\}_{i=1,\dots,d}$. Since $Q_k\cup Q_l=(Q_k\cup Q_l)\cup(Q_k\cap Q_l)$ as multisets, the sets $Q'_i$ also satisfy the multiset equality $\bigcup_iQ'_i=\bigcup_iQ^{\ell}_i$. Moreover, the number of pairs incomparable in the inclusion order among $Q'_i$ is lower than among sets $Q_i$. By the inductive assumption, the inequality $(\ref{eq0})$ holds for sets $Q'_i$. Combining it with submodularity of the rank function
\begin{equation*}
r(A\cap Q_k)+r(A\cap Q_l)\geq r(A\cap(Q_k\cup Q_l))+r(A\cap(Q_k\cap Q_l)),
\end{equation*}
we get inequality $(\ref{eq0})$ for sets $Q_i$. This completes the inductive step.
\end{proof}

\section{Applications}\label{Section3}

A family $\mathcal{B}$ of subsets of a finite set $E$, just from the definition, forms a set of bases of a matroid if it is non-empty, and if for every $B_1,B_2\in\mathcal{B}$ and $e\in B_1\setminus B_2$ there exists $f\in B_2\setminus B_1$, such that $B_1\cup f\setminus e\in\mathcal{B}$. 

In this case a stronger property holds. For every bases $B_1,B_2$ and $e\in B_1\setminus B_2$ there exists $f\in B_2\setminus B_1$, such that both $B_1\cup f\setminus e$ and $B_2\cup e\setminus f$ are bases. It is called \emph{symmetric exchange property}, and was discovered by Brualdi \cite{Br69}.

Surprisingly, even more is true. One can exchange symmetrically not only single elements, but also subsets. This property is known as \emph{multiple symmetric exchange}. We demonstrate usefulness of Theorem \ref{main} by using it as a tool to give easy proofs of multiple symmetric exchange property and its generalizations.

\begin{proposition} 
Let $B_1$ and $B_2$ be bases of a matroid $M$. Then for every $A_1\subset B_1$ there exists $A_2\subset B_2$, such that $(B_1\setminus A_1)\cup A_2$ and $(B_2\setminus A_2)\cup A_1$ are both bases.
\end{proposition}

\begin{proof} 
Observe that by adding parallel elements to the elements of the intersection $B_1\cap B_2$ we can restrict to the case when bases $B_1$ and $B_2$ are disjoint. 

When bases $B_1,B_2$ are disjoint, then restrict matroid $M$ to their union. Let $\ell$ be a lists size function such that $\ell\vert_{B_1}\equiv 1$, and $\ell\vert_{B_2}\equiv 2$. It is easy to check that condition $(1)$ of Theorem \ref{main} is satisfied. 

Let $L$ be a list assignment of size $\ell$ which assigns list $\{1\}$ to elements of $A_1$, list $\{2\}$ to elements of $B_1\setminus A_1$, and list $\{1,2\}$ to elements of $B_2$. By Theorem \ref{main} there exists a proper coloring from these lists. Denote by $C_1$ elements colored with $1$, and by $C_2$ those colored with $2$. Now $A_2=C_2\cap B_2$ is a good choice, since sets $(B_1\setminus A_1)\cup A_2=C_2$ and $(B_2\setminus A_2)\cup A_1=C_1$ are independent. 
\end{proof}

Multiple symmetric exchange property can be slightly generalized. Instead of having a partition of one of bases into two parts we can have an arbitrary partition of it. We prove that for any such partition there exists a partition of the second basis which is consistent in two different ways.

\begin{proposition} 
Let $A$ and $B$ be bases of a matroid $M$. Then for every partition $B_1\sqcup\dots\sqcup B_k=B$ there exists a partition $A_1\sqcup\dots\sqcup A_k=A$, such that $(B\setminus B_i)\cup A_i$ are bases for all $1\leq i\leq k$.
\end{proposition}

\begin{proof}
As before we can assume that bases $A,B$ are disjoint, and restrict to their union. Consider matroid $M'$ equal to $M$, where elements of $B$ have $k-1$ parallel copies. So the ground set of $M'$ equals to $B^1\sqcup\dots\sqcup B^{k-1}\sqcup A$ . Let $\ell$ be a lists size function such that $\ell\vert_{B^i}\equiv k-1$, and $\ell\vert_{A}\equiv k$. It is easy to check that condition $(1)$ of Theorem \ref{main} is satisfied. 

Let $L$ be a list assignment of size $\ell$ which assigns list $\{1,\dots,k\}\setminus\{i\}$ to elements of $B_i$ and all its copies, and list $\{1,\dots,k\}$ to elements of $A$. By Theorem \ref{main} there exists a proper coloring from these lists. Denote by $C_i$ elements colored with $i$. Now $A_i=C_i\cap A$ is a good partition, since sets $(B\setminus B_i)\cup A_i$ are independent. 
\end{proof}

\begin{proposition} 
Let $A$ and $B$ be bases of a matroid $M$. Then for every partition $B_1\sqcup\dots\sqcup B_k=B$ there exists a partition $A_1\sqcup\dots\sqcup A_k=A$, such that $(A\setminus A_i)\cup B_i$ are all bases for $1\leq i\leq k$.
\end{proposition}

\begin{proof}
We assume that bases $A,B$ are disjoint, and restrict to their union. Consider matroid $M'$ equal to $M$, where elements of $A$ have $k-1$ parallel copies. So the ground set of $M'$ equals to $A^1\sqcup\dots\sqcup A^{k-1}\sqcup B$ . Let $\ell$ be a lists size function such that $\ell\vert_{A^i}\equiv k$, and $\ell\vert_{B}\equiv 1$. It is easy to check that condition $(1)$ of Theorem \ref{main} is satisfied. 

Let $L$ be a list assignment of size $\ell$ which assigns list $\{i\}$ to elements of $B_i$, and list $\{1,\dots,k\}$ to elements of $A$ and all its copies. By Theorem \ref{main} there exists a proper coloring from these lists. Denote by $C_i$ elements colored with $i$. Now let $A_i$ contain all $a\in A$ such that no copy of $a$ is colored with $i$. Then sets $(A\setminus A_i)\cup B_i$ are independent, so it is a good partition. 
\end{proof}


\end{document}